\documentclass[12pt,twoside]{article}
\usepackage{a4}
\usepackage{amssymb,amsmath,amsthm,latexsym}
\usepackage{amsfonts}
\usepackage{amsfonts}
\usepackage{graphicx}
\usepackage{amsmath, amsfonts}
\usepackage{amssymb, graphicx}
\usepackage{amscd}
\usepackage{textcomp}
\usepackage{palatino}
\usepackage{xcolor}
\usepackage[colorlinks=true,linkcolor=red,citecolor=red]{hyperref}
\newtheorem{theorem}{Theorem}[section]

\newtheorem{corollary}[theorem] {Corollary}

\newtheorem{lemma} [theorem]{Lemma}

\newtheorem{proposition}[theorem]{Proposition}

\allowdisplaybreaks

\newcommand{\Q}{{\mathbb Q}}

\usepackage{indentfirst}

\vspace{5cm}

\begin{document}
	\label{'ubf'}  
	\setcounter{page}{1} 
	\markboth {\hspace*{-9mm} \centerline{\footnotesize
			Transcendental nature of $p$-adic digamma values}}
	{\centerline{\footnotesize 
			Tapas Chatterjee and Sonam Garg } \hspace*{-9mm}}
	
	\vspace*{-2cm}
	\begin{center}
		{{\textbf{Transcendental nature of $p$-adic digamma values
			}}\\
			\vspace{.2cm}
			\medskip
			{\sc Tapas Chatterjee\footnote{Research of the first author is partly supported by the core research grant CRG/2023/000804 of the Science and Engineering Research
Board of DST, Government of India.} }\\
			{\footnotesize  Department of Mathematics,}\\
			{\footnotesize Indian Institute of Technology Ropar, Punjab, India.}\\
			{\footnotesize e-mail: {\it tapasc@iitrpr.ac.in}}
			
			\medskip
			{\sc Sonam Garg\footnote{Research of the second author is supported by University Grants Commission (UGC), India under File No.: 972/(CSIR-UGC NET JUNE 2018).} }\\
			{\footnotesize Department of Mathematics, }\\
			{\footnotesize Indian Institute of Technology Ropar, Punjab, India.}\\
			{\footnotesize e-mail: {\it 2018maz0009@iitrpr.ac.in}}
			\medskip}
	\end{center}
	
	\thispagestyle{empty} 
	\vspace{-.4cm}
	\hrulefill
	\noindent
	\begin{abstract}  
		{\footnotesize }
		For a fixed prime $p$, Murty and Saradha (2008) studied the transcendental nature of special values of the $p$-adic digamma function, denoted as $\psi_p(r/p)+ \gamma_p$. 
		This research was later extended by Chatterjee and Gun in 2014, who investigated the case of $\psi_p(r/p^n)+ \gamma_p$, for any integer $n>1$. In this article, we generalize their results for distinct prime powers and explore the transcendental nature of the $p$-adic digamma values, with at most one exception.
		
		Further, we investigate the multiplicative independence of cyclotomic numbers satisfying certain conditions. Using this, we prove the transcendental nature of $p$-adic digamma values corresponding to $\psi_p(r/pq)+ \gamma_p$, where $p, q$ are distinct primes. 
	\end{abstract}
	\hrulefill
	
	
	
	\vspace{-.37cm}
	
	\section{Introduction} \label{INT}
	
	The gamma function $\Gamma$ is yielded by the following improper integral:
	\begin{align*}
	\Gamma(z) = \int_{0}^{\infty}{e^{-x} x^{z-1}dx}, \hspace{6mm} \Re(z)> 0.
	\end{align*}
	Except for the non-positive integers, where it has a simple pole with residue $(-1)^n/n!$, it is meromorphically extended to the entire complex plane. The digamma function, which is represented by the symbol $\psi$, is the logarithmic derivative of the gamma function, that is:
	\begin{align*}
	\psi(x) = \frac{d}{dx}(\log \Gamma(x)) = \frac{\Gamma^{\prime}(x)}{\Gamma(x)}.
	\end{align*}
	
	In 1813, Gauss gave the formula for the digamma function at rational argument $r/q$ for $q >1 $ and $1 \leq r < q$, with $(r, q) =1$ by: 
	\begin{align*}
	\psi(r/q) = -\gamma -\log q + \displaystyle\sum_{b=1}^{q-1}{e^{-2\pi ibr/q} \log(1- e^{2\pi ib/q})},
	\end{align*}
	where it can be observed that $\psi(1) = -\gamma$, the famous Euler's constant. Murty and Saradha investigated these functions in 2007 and published results on the transcendence of a specific family of digamma values in \cite{MS1}, which were extended by Chatterjee and Gun in \cite{TCS}. Each of these classical number theoretic functions has various analogues in different areas such as $p$-adic and $q$-series with a number of nice properties. Many interesting transcendence results related to $q$-analogues of these classical functions such as Riemann zeta function, Euler's constant, etc. are studied in \cite{TCG, TCG1}.  Thus similar results ought to be anticipated in the $p$-adic case as well.  Many mathematicians have worked in this area and come up with some promising results. To begin, Morita in \cite{DB} introduced the $p$-adic analogue of gamma function, $\Gamma_p$, for all the natural numbers, which is defined as follows:
	\begin{align*}
	\Gamma_p(n) = (-1)^n \displaystyle \prod_{\substack{1 \leq t \leq n \\ p\nmid t}} t 
	\end{align*} 
	and extended it to a continuous function on $\mathbb{Z}_p$. Then in \cite{JD}, Diamond subsequently introduced the $p$-adic counterparts of the digamma function and the Euler's constant, which satisfy the several nice properties as in the classical case. Diamond also discussed two different approaches to the $p$-adic analogue of log$\Gamma(x)$. One of the approach is to change the functional equation and the other is to construct a sequence of functions $H_N$, which is locally holomorphic on $\mathbb{C}_p$. These sequence of functions are defined as follows:
	\begin{align*}
	H_N(x)= \displaystyle \lim_{k \rightarrow \infty} \frac{1}{p^k} \sum_{n=0}^{p^k-1}f_N(x+n), ~~~~\text{for}~~~~ N=1,2, \ldots,
	\end{align*}
	where
	\begin{align*}
	f_N(x)= \begin{cases} x\log(x) - x, & \text{if}~~ \nu_p(x)<N\\
	0, & \text{if}~~ \nu_p(x) \geq N,
	\end{cases}
	\end{align*} 
	where $\nu_p(x)$ is the $p$-adic valuation. Each $f_N$ is locally analytic on $\mathbb{C}_p$, so each $H_N$ is also locally analytic on $\mathbb{C}_p$. Also, $H_N$ satisfies the following relation:
	\begin{align*}
	H_N(x+1)= \begin{cases} H_N(x) + \log(x), & \text{if}~~ \nu_p(x)<N\\
	H_N(x), & \text{if}~~ \nu_p(x) \geq N.
	\end{cases}
	\end{align*}
	The derivative of $p$-adic analogue of log$\Gamma(x)$ function is known as $p$-adic digamma function $\psi_p(x)$ and is given by the expression:
	\begin{align*}
	\psi_p(x) = \displaystyle\lim_{k\rightarrow\infty}\frac{1}{p^k}\sum_{n=0}^{p^k -1} \log_p(x+n), 
	\end{align*}
	for any $x\in \mathbb{C}_p$. The $p$-adic analogue of Euler-Briggs-Lehmer constant for $r, q \in \mathbb{Z}$ with $q\geq 1$ and $\nu_p(r/q)<0$ is given by:
	\begin{align*}
	\gamma_p(r,q) = -\displaystyle\lim_{k\rightarrow\infty}\frac{1}{qp^k}\sum_{\substack{m=0\\ m \equiv r(\rm mod~q)}}^{qp^k-1}{\log_p m} 
	\end{align*}
	and when $\nu_p(r/q)\geq 0$, write $q = p^k q_1$ with $(p,q_1)=1$, then
	\begin{align*}
	\gamma_p(r,q)=\frac{p^{\varphi(q_1)}}{p^{\varphi(q_1)}-1}\displaystyle\sum_{n\in N(r,q)}{\gamma_p(r+nq, p^{\varphi(q_1)}q)},
	\end{align*}
	where $N(r,q) = \{n: 0\leq n < p^{\varphi(q_1)}, nq + r \not\equiv 0$ (mod $p^{\varphi(q_1)+k})\}$. Also, we have:
	\begin{align*}
	\gamma_p = \gamma_p(0,1) = -\frac{p}{p- 1}\displaystyle\lim_{k\rightarrow\infty}\frac{1}{p^k}\sum_{\substack{m=1\\ (m,p)=1}}^{p^k - 1}{\log_p m}.
	\end{align*}
	
	Like in the classical case, the $p$-adic analogue of Gauss theorem in $\mathbb{C}_p$ is given as:
	\begin{align}
	\psi_p(r/f) = -\log f -\gamma_p + \displaystyle \sum_{a=1}^{f-1}{\zeta^{-ar} \log_p (1- \zeta^a)}, \label{EQ7}
	\end{align}
	where $r, f \in \mathbb{Z}^{+}$, $r < f$ and $\nu_p(r/f) <0$. However, for $\nu_p(r/f) \geq 0$ and any $\mu$ such that $p^{\mu} \equiv 1 (\bmod f^*)$, where $f = p^kf^*$ with $(p,f^*)=1$, we have the following relation:
	\begin{align}
	\frac{p^{\mu}}{p^{\mu}-1}H^{\prime}_{\mu}(r/f)= -\log f -\gamma_p + \displaystyle \sum_{a=1}^{f-1}{\zeta^{-ar} \log_p (1- \zeta^a)}.\label{EQ6}
	\end{align}
	
	Furthermore, for $U_1 = B(1,1)= \{x \in \mathbb{C}_p : |x - 1|
	_p < 1 \}$, the $p$-adic logarithm of $x \in U_1$ is defined as: 
	\begin{align*}
	\log_p (x) = \log_p (1 + (x-1)) = \displaystyle\sum_{n=1}^{\infty}{(-1)^{n+1} \frac{(x-1)^n}{n}}
	\end{align*}
	and it can be extended to the whole $\mathbb{C}_p^{\times}$ as every element $\beta \in \mathbb{C}_p^{\times}$ can be uniquely written as: 
	\begin{align*}
	\beta = p^r \omega x,
	\end{align*} 
	where $r \in \mathbb{Q}$, $x \in U_1$, and $\omega$ is a root of unity of order prime to $p$ and thus one defines: 
	\begin{align*}
	\log_p(\beta) = \log_p(x).
	\end{align*}
	As a consequence of this, the $p$-adic logarithm is zero on roots of unity and thus
	\begin{align} 
	\log_p(1- \zeta_q^{-t}) = \log_p(1- \zeta_q^{t}).\label{EQ1}
	\end{align} 
	
	In \cite{MS2}, Murty and Saradha established numerous conclusions  involving the values of the $p$-adic digamma function, $\psi_p(r/p)$, for $1 \leq r < p$, which were later extended by Chatterjee and Gun in \cite{TCS} with the following theorem:
	
	\begin{theorem} {\textbf{(Chatterjee and Gun)}} \label{TH1}
		Fix an integer $n>1$. At most one element of the following set:
		\begin{align*}
		\{\psi_p(r/p^n) + \gamma_p : 1 \leq r < p^n, (r,p) =1\}
		\end{align*}
		is algebraic. Moreover, $\psi_p(r/p) + \gamma_p$ are distinct when $1 \leq r < p/2$.
	\end{theorem}
	
	In this article, we further extend the above result for the set of rational primes $\mathcal{P}$ as follows:
	
	\begin{theorem} \label{TH2}
		Let $p$ be a prime and $n>1$ be an integer. Consider the sets $S_1$ and $S_2$ where
		\begin{align*}
		S_1&= \{\psi_p(r/p^n) + \gamma_p : 1 \leq r < p^n,  (r,p) = 1 \} ~~~ \text{and} \\
		S_2 &= \Bigg\{\frac{p^{\mu}}{p^{\mu}-1}H^{\prime}_{\mu}(r/q^n) +\gamma_p : 1 \leq r < q^n, (r,q) = 1, q\neq p, q \in \mathcal{P} \Bigg\},
		\end{align*}
		where $\mu$ as in Eq. (\ref{EQ6}). Then, all the elements of $S_1 \cup S_2$ are transcendental with at most one exception. Moreover, the numbers $\frac{p^{\mu}}{p^{\mu}-1}H^{\prime}_{\mu}(r/q) +\gamma_p$ are distinct, when $1 \leq r < q/2$ and $q \in \mathcal{P}$.
	\end{theorem}
	
	Further, we proceed to establish the result for the product of two distinct primes, wherein these primes satisfy \textbf{Property} $I$, which is given as follows:
	
	\textbf{Property $I$:} Let $m$ be a natural number such that $m=p_1^{\alpha_1}p_2^{\alpha_2}$, with $(\alpha_1, \phi (p_2^{\alpha_2})) = 1 = (\alpha_2, \phi (p_1^{\alpha_1}))$, where $p_1$, $p_2$ are odd primes, $\alpha_1$, $\alpha_2 \in \mathbb{N}$, and satisfies the following:
	\begin{enumerate}
		\item $ p_1 \equiv p_2 \equiv 3\pmod 4$: $p_1$ and $p_2$ are semi-primitive roots $\bmod~p_2^{\alpha_2}$ , $\bmod~p_1^{\alpha_1}$, respectively, or
		\item $p_1$, $p_2$ are primitive roots $\bmod~p_2^{\alpha_2}$ and $\bmod~p_1^{\alpha_1}$, respectively.
	\end{enumerate}
	
	\textbf{Property $II$:} Let $\mathcal{M}$ be a finite set of natural numbers with $\vert \mathcal{M} \vert = n$, containing pairwise co-prime integers $m_i$, where $1 \leq i \leq n$ such that $m_i$ satisfies \textbf{Property} $I$. 
	
	Let $\mathcal{J}$ consists of prime factors of $\{m_i\}_{i=1}^n$, where $m_i \in \mathcal{M}$. The theorems are then stated as follows: 
	
	\begin{theorem} \label{TH6}
		Let $p, ~q \in \mathcal{J}$ be any two primes such that $m =pq \in \mathcal{M}$. Then, the elements of the following set: 
		$$S_3 = \{\psi_p(r/pq) + \gamma_p : 1 \leq r < pq,  (r,pq) = 1\}$$
		are transcendental with at most one exception. Moreover, the numbers $\psi_p(r/pq) + \gamma_p$ are distinct when $1 \leq r < pq/2$ and $(r,pq) = 1$.
	\end{theorem}
	
	\begin{theorem} \label{TH3}
		Let $p$ be a prime. Then, the elements of the following set:
		\begin{align*}
		S_4 = \Bigg\{\frac{p^{\mu}}{p^{\mu}-1}H^{\prime}_{\mu}(r/m_i) +\gamma_p : 1 \leq r < m_i, ~1 \leq i \leq n,~ (r,m_i) = 1, p \nmid m_i, ~m_i \in \mathcal{M}\Bigg\},
		\end{align*}
		where $\mu$ as in Eq. (\ref{EQ6}), are transcendental with at most one exception.
	\end{theorem}
	
	Theorem \ref{TH6} and Theorem \ref{TH3} gives an important corollary which is stated as follows:
	\begin{corollary}\label{C1}
		All the elements of $S_3 \cup S_4$ are transcendental with at most one exception.
	\end{corollary}
	
	So far we have discussed the case where the composite numbers $q \not\equiv 2(\bmod ~4)$. We now investigate the scenario where $q \equiv 2(\bmod~ 4)$, employing Proposition \ref{P5} put forth by Chatterjee and Dhillon in \cite{CD3}. The theorem is stated as follows: 
	\begin{theorem}\label{TH7}
		Let $p$ be any prime and $q$ be an element of $\mathcal{H}$, where elements of $\mathcal{H}$ satisfy conditions of Proposition \ref{P5}. Then, we have the following statements:
		\begin{enumerate}
			\item If $p\mid q$, then the set of elements $$S_5=\{\psi_p(r/q) + \gamma_p : 1 \leq r < q,  (r,q) = 1 \}$$ are transcendental with at most one exception.
			\item If $p \nmid q$, then the set of elements $$S_6=\Bigg\{\frac{p^{\mu}}{p^{\mu}-1}H^{\prime}_{\mu}(r/q) +\gamma_p : 1 \leq r < q, (r,q) = 1 \Bigg\}$$ are transcendental with at most one exception.
		\end{enumerate}
	\end{theorem}
	
	\section{Notations and Preliminaries}
	
	This section addresses all of the necessary notations and known results that are required to prove the results. To begin with, $\mathbb{Q}$ and $\overline{\mathbb{Q}}$ denote the field of rational numbers and the field of algebraic numbers, respectively. 
	For prime $p$, $\mathbb{Q}_p$ denotes the $p$-adic completion of $\mathbb{Q}$, $\mathbb{C}_p$ denotes the completion of algebraic closure of $\mathbb{Q}_p$, and fixing an embedding of $\overline{\mathbb{Q}}$ into $\mathbb{C}_p$, the elements of $\mathbb{C}_p \setminus \overline{\mathbb{Q}}$ are transcendental numbers. Also, $\nu_p$ denotes the $p$-adic valuation in $\mathbb{C}_p$ with $\nu_p(p) = 1$ and $\mid \cdot \mid_p$ denotes the p-adic norm with $\mid p \mid_p = 1/p$.
	
	Further, Diamond in \cite{JD} gave the following theorem relating the $p$-adic analogues which we shall need to prove our result:
	\begin{theorem}{\textbf{(Diamond)}} \label{TH4}
		If $q > 1$ and $\zeta_q$ is a primitive $q$-th root of unity, then
		\begin{align*}
		q\gamma_p(r,q) = \gamma_p - \displaystyle\sum_{a=1}^{q-1}{\zeta^{-ar}_q \log_p(1-\zeta_q^a)}.
		\end{align*}
	\end{theorem}
	
	Next, we look at R. Kaufman's $p$-adic counterpart, much as in the classical case, where Baker's theorem was critical in establishing statements about classical logarithms of complex numbers.
	
	\begin{theorem}{\textbf{(R. Kaufman)}}
		Let $\alpha_1,\alpha_2,\ldots,\alpha_m$ be fixed algebraic numbers that are multiplicatively independent over $\mathbb{Q}$ with height at most $h$. Let $\beta_0,\beta_1,\ldots,\beta_m$ be arbitrary algebraic numbers with height at  most $H$ (assumed greater than $1$) and $\beta_0 \neq 0$. There exists a constant $c_1 >0$ which depends only on the degree of the number field generated by $\alpha_1,\alpha_2, \ldots ,\alpha_m,\beta_0,\beta_1,\ldots,\beta_m$ such that the following holds. Let $K = \mathbb{Q}(\alpha_1,\alpha_2, \ldots ,\alpha_m,\beta_0,\beta_1,\ldots,\beta_m)$ and $|\alpha_i -1|_p < p^{-c_1}$, for $1 \leq i \leq m$. Then, 
		\begin{align*}
		|\beta_0 + \beta_1\log_p\alpha_1 +\cdots + \beta_m \log_p \alpha_m|_p > p^{-c \log_pH},
		\end{align*} 
		where $c$ is a constant depending only on $p$, $h$, $m$ and $\left[ K: \mathbb{Q}\right]$.
	\end{theorem}
	
	Murty and Saradha in \cite{MS2} presented the following conclusion as a result of this theorem:
	\begin{theorem} \label{TH5}
		Suppose that $\alpha_1,\alpha_2,\ldots,\alpha_m$ are non-zero algebraic numbers that are multiplicatively independent over $\mathbb{Q}$ and $\beta_1,\beta_2,\ldots,\beta_m$ are arbitrary algebraic numbers (not all zero). Further, suppose that
		\begin{align*}
		\mid \alpha_i -1 \mid_p < p^{-c},~~~ \text{for}~~~ 1 \leq i \leq m,
		\end{align*}
		where $c$ is a constant which depends only on the degree of the number field generated by $\alpha_1,\alpha_2, \ldots ,\alpha_m,\beta_1,\beta_2,\ldots,\beta_m$. Then,
		\begin{align*}
		\beta_1\log_p\alpha_1 +\cdots + \beta_m \log_p \alpha_m
		\end{align*} 
		is transcendental.
	\end{theorem}
	
	In addition, Chatterjee and Gun \cite{TCS} made the following important claim on the multiplicative independence of cyclotomic numbers, which played a vital role in their theorems:
	\begin{proposition} \label{P1}
		For $p_i \in \mathcal{P}$, let $q_i = p_i^{m_i}$, where $m_i \in \mathbb{N}$ and $\zeta_{q_i}$ be a primitive $q_i$-th root of unity. Then, for any finite subset $\mathcal{K}$ of $\mathcal{P}$, the numbers
		\begin{align*}
		1-\zeta_{q_i}, \frac{1-\zeta_{q_i}^{a_i}}{1-\zeta_{q_i}},~~~ \text{where}~~~ 1<a_i<\frac{q_i}{2}, (a_i, q_i)=1,~~~ \text{and}~~~ p_i \in \mathcal{K}
		\end{align*}
		are multiplicatively independent.
	\end{proposition}
	
	Then, using Proposition \ref{P1} along with Theorem \ref{TH5}, they proved the following two important lemmas in \cite{TCS} which were used to prove Theorem \ref{TH1}:
	\begin{lemma} \label{L1}
		Let $\mathcal{K}$ be any finite subset of $\mathcal{P}$. For $q \in \mathcal{K}$ and $1 < a < q/2$, let $s_q$, $t_{q_a}$ be arbitrary algebraic numbers, not all zero. Further, let $t_{q_a}$ be not all zero when $p \in \mathcal{K}$. Then,
		\begin{align*}
		\sum_{q \in \mathcal{K}}{s_q \log_p(1-\zeta_q)}  + \sum_{\substack{q \in \mathcal{K},\\ 1 < a< q/2}}{t_{q_a} \log_p \left(\frac{1-\zeta^a_q}{1-\zeta_q}\right)}
		\end{align*}
		is transcendental.
	\end{lemma}
	
	\begin{lemma} \label{L2}
		Let $q_1$, $q_2$ be two distinct prime numbers and $ 1 \leq r_i < q_i$, for $i=1,2$. Then,
		\begin{align*}
		\sum_{b=1}^{q_2 - 1}{\zeta^{-br_2}_{q_2}} \log_p(1-\zeta^b_{q_2}) - \sum_{a=1}^{q_1 - 1}{\zeta^{-ar_1}_{q_1}} \log_p(1-\zeta^a_{q_1})
		\end{align*}
		is transcendental.
	\end{lemma}
	
	In the later part of this section, we generalize Lemma \ref{L1}. To accomplish this, we require a necessary and sufficient condition for the multiplicative independence of cyclotomic units for prime powers and a product of distinct primes. This condition is provided by the following proposition established by Pei and Feng in \cite{PF}:
	\begin{proposition} \label{P3}
		For a composite number $q \not \equiv 2 ( \text{ mod } 4),$ the system 
		\begin{equation*}\label{27}
		\bigg \{\frac{1-\zeta_q^h}{1-\zeta_q}\big | \ (h,q)=1, \ 2 \leq h < q/2\bigg\}
		\end{equation*}
		of cyclotomic units of field $\Q(\zeta_q)$  is independent if and only if one of the following conditions 
		are satisfied (here $\alpha_0 \geq 3; \alpha_1, \alpha_2, \alpha_3 \geq 1; p_1, p_2 ,p_3$ are odd primes): 
		\begin{enumerate}
			\item $q=4p_1^{\alpha_1}$; and
			\begin{enumerate}
				\item 2 is a primitive root $\bmod~p_1^{\alpha_1}$; or
				\item 2 is a semi-primitive root $\bmod~p_1^{\alpha_1}$ and $p_1 \equiv 3 \text{ (mod } 4)$.
			\end{enumerate}
			\item $q=2^{\alpha_0}p_1^{\alpha_1}$; the order of $p_1\bmod~2^{\alpha_0})$ is $2^{\alpha_0-2}$, $2^{\alpha_0-3}p_1 \not \equiv-1( \bmod~2^{\alpha_0}),$ and
			\begin{enumerate}
				\item 2 is a primitive root $\bmod~p_1^{\alpha_1}$; or
				\item 2 is a semi-primitive root $\bmod~p_1^{\alpha_1}$ and $p_1 \equiv 3 (\bmod~4)$.
			\end{enumerate}
			\item $q=p_1^{\alpha_1}p_2^{\alpha_2}$; and 
			\begin{enumerate}
				\item when $p_1 \equiv p_2 \equiv 3 (\bmod~4)$: $p_1$ is a semi-primitive root $\bmod~p_2^{\alpha_2}$ and $p_2$ is a semi-primitive root $\bmod~p_1^{\alpha_1},$ or vice versa.
				\item otherwise: $p_1$ and $p_2$ are primitive root $\bmod~p_2^{\alpha_2}$ and $\bmod~p_1^{\alpha_1}$, respectively.
			\end{enumerate}
			\item $q=4p_1^{\alpha_1}p_2^{\alpha_2}; (p_1-1,p_2-1)=2$ and
			\begin{enumerate}
				\item when $p_1 \equiv p_2 \equiv 3 (\bmod~4)$: 2 is a primitive root for one $p$ and a semi-primitive root for another $p$; $p_1$ is primitive root $\bmod~2p_2^{\alpha_2}$ and $p_2$ is a semi-primitive root $\bmod~2p_1^{\alpha_1}$ or vice versa.
				\item when $p_1 \equiv 1, p_2 \equiv 3 (\bmod~4)$: 2 is a primitive root $\bmod~p_2^{\alpha_2}$; $p_1$ and $p_2$ are primitive root $\bmod~p_2^{\alpha_2}$ and $\bmod~p_1^{\alpha_1}$, respectively.
			\end{enumerate}
			\item $q=p_1^{\alpha_1}p_2^{\alpha_2}p_3^{\alpha_3}$; $p_1 \equiv  p_2 \equiv p_3 \equiv 3 (\bmod~4)$: $(p^i-1)/2 ~~  ( 1 \leq i \leq 3)$ are co-prime to each other; and 
			\begin{enumerate}
				\item $p_1, p_2, p_3$ are primitive root $\bmod~p_2^{\alpha_2}$, $\bmod~p_3^{\alpha_3}$, $\bmod~p_1^{\alpha_1}$, respectively and semi-primitive root $\bmod~p_3^{\alpha_3}$, $\bmod~p_1^{\alpha_1}$, $\bmod~p_2^{\alpha_2}$, respectively.
			\end{enumerate}
		\end{enumerate}
	\end{proposition}
	
	In 2020, Chatterjee and Dhillon extended Proposition \ref{P1} by using \textbf{Property} $I$ and \textbf{Property} $II$ (see \cite{CD2}), which is given as Proposition \ref{P2}.
	
	\begin{proposition} \label{P2}
		Let $\{m_i\}_{i=1}^n$ be a set of natural numbers that satisfies \textbf{Property} $II$ and $\zeta_{m_i}$ be a primitive $m_i$-th root of unity. Then, the numbers
		\begin{align*}
		1-\zeta_{p_{i}},1-\zeta_{q_{i}}, \frac{1-\zeta_{m_i}^{a_i}}{1-\zeta_{m_i}},
		\end{align*}
		where $m_i = p_iq_i$, $1<a_i<\frac{m_i}{2}$ with $(a_i, m_i)=1$ and $1 \leq i \leq n$ are multiplicatively independent.
	\end{proposition}
	
	Further, motivated by the idea used to prove Proposition \ref{P2}, we have the following proposition which plays a pivotal role in the proof of Theorem \ref{TH6}:
	
	\begin{proposition} \label{P4}
		Let $\{m_i\}_{i=1}^n$ be a set of natural numbers that satisfies \textbf{Property} $II$ and $\zeta_{m_i}$ be a primitive $m_i$-th root of unity. Then, the following numbers:
		\begin{align*}
		1-\zeta_{p_{i}}, 1-\zeta_{q_{i}}, \frac{1-\zeta_{m_i}^{a_i}}{1-\zeta_{m_i}}, \frac{1-\zeta_{p_{i}}^{b_i}}{1-\zeta_{p_{i}}}, \frac{1-\zeta_{q_{i}}^{c_i}}{1-\zeta_{q_{i}}},
		\end{align*}
		where $m_i = p_iq_i$, $1<a_i<\frac{m_i}{2}$ with $(a_i, m_i)=1$, $1<b_i<\frac{p_{i}}{2}$, $1<c_i<\frac{q_{i}}{2}$, and $1 \leq i \leq n$ are multiplicatively independent.
	\end{proposition}
	\begin{proof}
		Let there be integers $\alpha_{i}$, $\beta_{i}$, $\beta_{b_i}$, $\beta_{c_i}$ and $\delta_{a_i}$, if possible such that
		\begin{align}
		&\prod_{1 \leq i \leq n}({1-\zeta_{p_{i}}})^{\alpha_{i}} \prod_{1 \leq i \leq n}({1-\zeta_{q_{i}}})^{\beta_{i}} \prod_{\substack{1<a_i < m_i/2\\ (a_i, m_i)=1 \\ 1 \leq i \leq n}}\Bigg(\frac{1-\zeta_{m_i}^{a_i}}{1-\zeta_{m_i}}\Bigg)^{\delta_{a_i}} \nonumber\\
		&\qquad \prod_{\substack{1<b_i < p_{i}/2 \\ 1 \leq i \leq n}}\Bigg(\frac{1-\zeta_{p_{i}}^{b_i}}{1-\zeta_{p_{i}}}\Bigg)^{\beta_{b_i}} \prod_{\substack{1<c_i < q_{i}/2 \\ 1 \leq i \leq n}}\Bigg(\frac{1-\zeta_{q_{i}}^{c_i}}{1-\zeta_{q_{i}}}\Bigg)^{\beta_{c_i}} = 1. \label{EQ2}
		\end{align}
		Considering the norm on both sides, for $A_{i}$, $B_{i}$ $\in \mathbb{N}$ we obtain:
		\begin{align*}
		\prod_{1 \leq i \leq n}p_{i}^{\alpha_{i} A_{i}} \prod_{1 \leq i \leq n}q_{i}^{\beta_{i} B_{i}} = 1.
		\end{align*}
		This imply that $\alpha_{i} = 0$ and $\beta_{i} = 0$, for $1 \leq i \leq n$ as $p_{i}'s$ and $q_{i}'s$ are distinct primes.
		Accordingly, Eq. (\ref{EQ2}) reduces to
		\begin{align*}
		\prod_{\substack{1<a_i < m_i/2\\ (a_i, m_i)=1 \\ 1 \leq i \leq n}}\Bigg(\frac{1-\zeta_{m_i}^{a_i}}{1-\zeta_{m_i}}\Bigg)^{\delta_{a_i}} \prod_{\substack{1<b_i < p_{i}/2 \\ 1 \leq i \leq n}}\Bigg(\frac{1-\zeta_{p_{i}}^{b_i}}{1-\zeta_{p_{i}}}\Bigg)^{\beta_{b_i}} \prod_{\substack{1<c_i < q_{i}/2 \\ 1 \leq i \leq n}}\Bigg(\frac{1-\zeta_{q_{i}}^{c_i}}{1-\zeta_{q_{i}}}\Bigg)^{\beta_{c_i}} = 1.
		\end{align*}
		Rewriting the aforementioned equation, we get:
		\begin{align}
		&\prod_{\substack{1<a_i < m_i/2\\ (a_i, m_i)=1 \\ 2 \leq i \leq n}}\Bigg(\frac{1-\zeta_{m_i}^{a_i}}{1-\zeta_{m_i}}\Bigg)^{\delta_{a_i}} \prod_{\substack{1<b_i < p_{i}/2 \\ 2 \leq i \leq n}}\Bigg(\frac{1-\zeta_{p_{i}}^{b_i}}{1-\zeta_{p_{i}}}\Bigg)^{\beta_{b_i}} \prod_{\substack{1<c_i < q_{i}/2 \\ 2 \leq i \leq n}}\Bigg(\frac{1-\zeta_{q_{i}}^{c_i}}{1-\zeta_{q_{i}}}\Bigg)^{\beta_{c_i}} \nonumber \\
		&\quad = \prod_{\substack{1<a_1 < m_1/2\\ (a_1, m_1)=1 }}\Bigg(\frac{1-\zeta_{m_1}^{a_1}}{1-\zeta_{m_1}}\Bigg)^{-\delta_{a_1}} \prod_{1<b_1 < p_{1}/2}\Bigg(\frac{1-\zeta_{p_{1}}^{b_1}}{1-\zeta_{p_{1}}}\Bigg)^{-\beta_{b_1}} \prod_{1<c_1 < q_{1}/2}\Bigg(\frac{1-\zeta_{q_{1}}^{c_1}}{1-\zeta_{q_{1}}}\Bigg)^{-\beta_{c_1}}. \label{EQ3}
		\end{align}
		Clearly, the right side of the aforementioned equation belongs to the number field $\mathbb{Q}(\zeta_{m_1})$, whereas the left side belongs to $\mathbb{Q}(\zeta_{r})$, where $$r= \displaystyle\prod_{2 \leq i \leq n}m_i.$$ 
		Also, it is known that $\mathbb{Q}(\zeta_{m_1}) \cap \mathbb{Q}(\zeta_{r}) = \mathbb{Q} $. As a result, the two sides of Eq. (\ref{EQ3}) are rational numbers having norm $1$ and hence
		\begin{align*}
		&\prod_{\substack{1<a_i < m_i/2\\ (a_i, m_i)=1 \\ 2 \leq i \leq n}}\Bigg(\frac{1-\zeta_{m_i}^{a_i}}{1-\zeta_{m_i}}\Bigg)^{\delta_{a_i}} \prod_{\substack{1<b_i < p_{i}/2 \\ 2 \leq i \leq n}}\Bigg(\frac{1-\zeta_{p_{i}}^{b_i}}{1-\zeta_{p_{i}}}\Bigg)^{\beta_{b_i}} \prod_{\substack{1<c_i < q_{i}/2 \\ 2 \leq i \leq n}}\Bigg(\frac{1-\zeta_{q_{i}}^{c_i}}{1-\zeta_{q_{i}}}\Bigg)^{\beta_{c_i}}\\
		&\quad =\prod_{\substack{1<a_1 < m_1/2\\ (a_1, m_1)=1 }}\Bigg(\frac{1-\zeta_{m_1}^{a_1}}{1-\zeta_{m_1}}\Bigg)^{-\delta_{a_1}} \prod_{1<b_1 < p_{1}/2}\Bigg(\frac{1-\zeta_{p_{1}}^{b_1}}{1-\zeta_{p_{1}}}\Bigg)^{-\beta_{b_1}} \prod_{1<c_1 < q_{1}/2}\Bigg(\frac{1-\zeta_{q_{1}}^{c_1}}{1-\zeta_{q_{1}}}\Bigg)^{-\beta_{c_1}}\\
		& \quad = \pm 1. 
		\end{align*}
		Consequently, squaring both the sides, we obtain:
		\begin{align*}
		&\prod_{\substack{1<a_i < m_i/2\\ (a_i, m_i)=1 \\ 2 \leq i \leq n}}\Bigg(\frac{1-\zeta_{m_i}^{a_i}}{1-\zeta_{m_i}}\Bigg)^{2\delta_{a_i}} \prod_{\substack{1<b_i < p_{i}/2 \\ 2 \leq i \leq n}}\Bigg(\frac{1-\zeta_{p_{i}}^{b_i}}{1-\zeta_{p_{i}}}\Bigg)^{2\beta_{b_i}} \prod_{\substack{1<c_i < q_{i}/2 \\ 2 \leq i \leq n}}\Bigg(\frac{1-\zeta_{q_{i}}^{c_i}}{1-\zeta_{q_{i}}}\Bigg)^{2\beta_{c_i}}\\
		& \quad =\prod_{\substack{1<a_1 < m_1/2\\ (a_1, m_1)=1 }}\Bigg(\frac{1-\zeta_{m_1}^{a_1}}{1-\zeta_{m_1}}\Bigg)^{-2\delta_{a_1}} \prod_{1<b_1 < p_{1}/2}\Bigg(\frac{1-\zeta_{p_{1}}^{b_1}}{1-\zeta_{p_{1}}}\Bigg)^{-2\beta_{b_1}} \prod_{1<c_1 < q_{1}/2}\Bigg(\frac{1-\zeta_{q_{1}}^{c_1}}{1-\zeta_{q_{1}}}\Bigg)^{-2\beta_{c_1}}\\
		& \quad = 1. 
		\end{align*}
		Now, consider
		\begin{align}
		\prod_{\substack{1<a_1 < m_1/2\\ (a_1, m_1)=1 }}\Bigg(\frac{1-\zeta_{m_1}^{a_1}}{1-\zeta_{m_1}}\Bigg)^{-2\delta_{a_1}} \prod_{1<b_1 < p_{1}/2}\Bigg(\frac{1-\zeta_{p_{1}}^{b_1}}{1-\zeta_{p_{1}}}\Bigg)^{-2\beta_{b_1}} \prod_{1<c_1 < q_{1}/2}\Bigg(\frac{1-\zeta_{q_{1}}^{c_1}}{1-\zeta_{q_{1}}}\Bigg)^{-2\beta_{c_1}} = 1. \label{EQ4}
		\end{align}
		Thus, we have:
		\begin{align*}
		\prod_{\substack{1<a_1 < m_1/2\\ (a_1, m_1)=1 }}\Bigg(\frac{1-\zeta_{m_1}^{a_1}}{1-\zeta_{m_1}}\Bigg)^{-2\delta_{a_1}} \prod_{1<b_1 < p_{1}/2}\Bigg(\frac{1-\zeta_{p_{1}}^{b_1}}{1-\zeta_{p_{1}}}\Bigg)^{-2\beta_{b_1}}  = \prod_{1<c_1 < q_{1}/2}\Bigg(\frac{1-\zeta_{q_{1}}^{c_1}}{1-\zeta_{q_{1}}}\Bigg)^{2\beta_{c_1}}.
		\end{align*}
		Note that the right-hand side of the above equation belongs to $\mathbb{Q}(\zeta_{q_{1}})$, while the left-hand side belongs to $\mathbb{Q}(\zeta_{m_{1}})$.  Furthermore, it is essential to emphasize that the left side should also belong to $\mathbb{Q}(\zeta_{q_{1}})$. Finally, using coprimality condition $(a_1, m_1) =1$, it follows that $\beta_{b_1} = 0$, for all $1 < b_1 < p_{1}/2 $. Once again, rewriting the Eq. (\ref{EQ4}) as:
		\begin{align*}
		\prod_{\substack{1<a_1 < m_1/2\\ (a_1, m_1)=1 }}\Bigg(\frac{1-\zeta_{m_1}^{a_1}}{1-\zeta_{m_1}}\Bigg)^{-2\delta_{a_1}} \prod_{1<c_1 < q_{1}/2}\Bigg(\frac{1-\zeta_{q_{1}}^{c_1}}{1-\zeta_{q_{1}}}\Bigg)^{-2\beta_{c_1}} = \prod_{1<b_1 < p_{1}/2}\Bigg(\frac{1-\zeta_{p_{1}}^{b_1}}{1-\zeta_{p_{1}}}\Bigg)^{2\beta_{b_1}}
		\end{align*}
		and proceeding in the similar way, we get $\beta_{c_1} = 0$, for all $1 < c_1 < q_{1}/2 $.
		Now from Eq. (\ref{EQ4}), we have:
		\begin{align*}
		\prod_{\substack{1<a_1 < m_1/2\\ (a_1, m_1)=1 }}\Bigg(\frac{1-\zeta_{m_1}^{a_1}}{1-\zeta_{m_1}}\Bigg)^{-2\delta_{a_1}} = 1.
		\end{align*}
		Then, by Proposition \ref{P3}, we have $\delta_{a_1} = 0$, for all $1<a_1 < m_1/2$ with $(a_1, m_1)=1$. Similarly, we get $\delta_{a_i} = 0$, for all $1<a_i < m_i/2$ with $(a_i, m_i)=1$, $\beta_{b_i} = 0$, for all $1 < b_i < p_{i}/2 $, and $\beta_{c_i} = 0$, for all $1 < c_i < q_{i}/2 $ and $1 \leq i \leq n$. This completes the proof.
	\end{proof}
	
	Note that the above proposition is the further generalization of Proposition \ref{P1} and Proposition \ref{P2}. Now, using Theorem \ref{TH5} and Proposition \ref{P4}, we have the following lemma which served as the extension of Lemma \ref{L1}:
	\begin{lemma} \label{L3}
		Assuming \textbf{Property} $II$ and let $\zeta_{m_i}$ be a primitive $m_i$-th root of unity. Let $r_{q_{i}}$,$u_{{q_i}_{b_i}}$, $t_{{m_i}_{a_i}}$ be arbitrary algebraic numbers, not all zero. Further, let $u_{{q_i}_{b_i}}$, $t_{{m_i}_{a_i}}$ be not all zero, when $p \in \mathcal{J}$. Then,
		\begin{align*}
		\sum_{q_{i} \in \mathcal{J}}{r_{q_{i}} \log_p(1-\zeta_{q_{i}})} + \sum_{\substack{q_{i} \in \mathcal{J},\\ 1 < b_i< q_{i}/2}}{u_{{q_i}_{b_i}} \log_p \left(\frac{1-\zeta_{q_{i}}^{b_i}}{1-\zeta_{q_{i}}}\right)} + \sum_{\substack{m_i \in \mathcal{M},\\ 1 < a_i< m_i/2}}{t_{{m_i}_{a_i}} \log_p \left(\frac{1-\zeta_{m_i}^{a_i}}{1-\zeta_{m_i}}\right)} 
		\end{align*}
		is transcendental.
	\end{lemma}
	
	\begin{proof}
		Let $\theta = \displaystyle \prod_{m_i \in \mathcal{M}}{m_i}$. For any $\alpha \in \mathbb{Z}[\zeta_\theta]$, with $(p, \alpha) = \mathbb{Z}[\zeta_{\theta}]$ and $K \in \mathbb{N}$, one has
		\begin{align*}
		|\alpha^T - 1|_p< p ^{-K}
		\end{align*} 
		for some $T \in \mathbb{N}$. By choosing $K$ sufficiently large and using Theorem \ref{TH5} and Proposition \ref{P4}, we get the desired result.
	\end{proof}
	
	Additionally, in 2020, Chatterjee and Dhillon introduced the following proposition in \cite{CD3}, offering the necessary and sufficient condition for the multiplicative independence of cyclotomic numbers when dealing with any composite number $q \equiv 2(\bmod ~4)$. 
	\begin{proposition} \label{P5}
		For any composite number $q \equiv 2(\bmod ~4)$, the system 
		\begin{align*}
		\Bigg\{\frac{1-\zeta_q^h}{1-\zeta_q}: (h,q) = 1, 1 < h < q/2 \Bigg\}
		\end{align*}
		is multiplicatively independent if and only if $q$ satisfies one of the following conditions:
		\begin{enumerate}
			\item $q = 2p^n$, where $p$ is an odd prime
			\item $q= 2m$, where $m$ satisfies condition III and V in Proposition \ref{P3}.
		\end{enumerate}
	\end{proposition}
	
	Let $\mathcal{H}$ be set of all those $q$ that satisfy the two conditions of Proposition \ref{P5}. A key part in the proof of Theorem \ref{TH7} is played by this proposition. 
	
	\section{Proofs of the Main Theorems}
	
	\begin{proof}[Proof of Theorem \ref{TH2}]
		Consider the sets
		\begin{align*}
		S_1&= \{\psi_p(r/p^n) + \gamma_p : 1 \leq r < p^n,  (r,p) = 1 \} ~~~ \text{and} \\
		S_2 &= \Bigg\{\frac{p^{\mu}}{p^{\mu}-1}H^{\prime}_{\mu}(r/q^n) +\gamma_p : 1 \leq r < q^n, (r,q) = 1, q\neq p, q \in \mathcal{P} \Bigg\},
		\end{align*}
		where $\mu$ as in (\ref{EQ6}).
		Let $S =  S_1 \cup S_2$. 
		Suppose that $a , b \in S$ are distinct and algebraic.
		Then, possibilities of $(a, b)$ are:
		\begin{enumerate}
			\item $(a, b) = ( \psi_p(r_1/p^n) +\gamma_p, \hspace{2mm} \psi_p(r_2/p^n) +\gamma_p)$.
			\item $(a, b) = \Big(\frac{p^{\mu_1}}{p^{\mu_1}-1}H^{\prime}_{\mu_1}(r_1/q_1^n) +\gamma_p, \hspace{2mm}  \frac{p^{\mu_2}}{p^{\mu_2}-1}H^{\prime}_{\mu_2}(r_2/q_2^n) +\gamma_p\Big)$.
			\item $(a, b) = ( \psi_p(r_1/p^n) +\gamma_p,\hspace{2mm} \frac{p^{\mu}}{p^{\mu}-1}H^{\prime}_{\mu}(r_2/q^n) +\gamma_p)$.
		\end{enumerate}
		The case where $a,b \in S_1$ has been proved in \cite{TCS}. Now, consider when  $a,b \in S_2$. So, using Eq. (\ref{EQ6}), we have:
		\begin{align*}
		&\frac{p^{\mu_1}}{p^{\mu_1}-1}H^{\prime}_{\mu_1}(r_1/q_1^n) +\gamma_p - \Big(\frac{p^{\mu_2}}{p^{\mu_2}-1}H^{\prime}_{\mu_2}(r_2/q_2^n) +\gamma_p\Big)\\
		& \quad = - \log_p q_1^n + \displaystyle\sum_{a=1}^{q_1^n - 1}{\zeta_{q_1^n}^{-ar_1}} \log_p (1- \zeta_{q_1^n}^{a}) + \log_p q_2^n - \sum_{t=1}^{q_2^n -1}{\zeta_{q_2^n}^{-tr_2}} \log_p (1- \zeta_{q_2^n}^t).
		\end{align*}
		This can be further simplified using Eq. (\ref{EQ1}) and the fact that for $p = q^n$ and $\zeta$ be the $q^n$-th primitive root of unity, we have $q= \prod ( 1- \zeta)$, where product runs over all the primitive $q^n$-th root of unity. So, we have:
		\begin{align*}
		&- \log_p q_1^n + \displaystyle\sum_{a=1}^{q_1^n - 1}{\zeta_{q_1^n}^{-ar_1}} \log_p (1- \zeta_{q_1^n}^{a}) + \log_p q_2^n - \sum_{t=1}^{q_2^n -1}{\zeta_{q_2^n}^{-tr_2}} \log_p (1- \zeta_{q_2^n}^t)\\
		&\quad = -n \log_p \Big(\displaystyle\prod_{\substack{b=1\\(b,q_1)=1}}^{q_1^n - 1}{(1- \zeta_{q_1^n}^b)\Big)} + \log_p(1- \zeta_{q_1^n}) + n \log_p \Big(\displaystyle\prod_{\substack{s=1\\(s,q_2)=1}}^{q_2^n - 1}{(1- \zeta_{q_2^n}^s)\Big)} - \log_p(1- \zeta_{q_2^n})\\
		&\quad \quad  + \sum_{1<a<q_1^n/2} (\zeta_{q_1^n}^{-ar_1} + \zeta_{q_1^n}^{ar_1})\log_p \Big(\frac{1-\zeta_{q_1^n}^a}{1-\zeta_{q_1^n}}\Big) - \sum_{1<t<q_2^n/2} (\zeta_{q_2^n}^{-tr_2} + \zeta_{q_2^n}^{tr_2})\log_p \Big(\frac{1-\zeta_{q_2^n}^t}{1-\zeta_{q_2^n}}\Big)\\
		& \quad = \log_p(1- \zeta_{q_1^n}) - n {\log_p(1-\zeta_{q_1^n})}- n {\displaystyle\sum_{\substack{b=2\\(b,q_1)=1}}^{q_1^n - 1}\log_p(1-\zeta_{q_1^n}^b)} - \log_p(1- \zeta_{q_2^n}) + n{\log_p(1-\zeta_{q_2^n})}\\ 
		&\quad \quad +n {\displaystyle\sum_{\substack{s=2\\(s,q_2)=1}}^{q_2^n - 1}\log_p(1-\zeta_{q_2^n}^s)} + \displaystyle\sum_{\substack{1<a<q_1^n/2\\(a,q_1)=1}} \alpha_a \log_p \Big(\frac{1-\zeta_{q_1^n}^a}{1-\zeta_{q_1^n}}\Big) - \displaystyle\sum_{\substack{1<t<q_2^n/2\\(t,q_2)=1}}\beta_t \log_p \Big(\frac{1-\zeta_{q_2^n}^t}{1-\zeta_{q_2^n}}\Big)\\
		& \quad = \delta \log_p(1- \zeta_{q_1^n}) + \eta \log_p(1- \zeta_{q_2^n}) + \sum_{\substack{1<a<q_1^n/2\\(a,q_1)=1}} \alpha_a^{\prime} \log_p \left(\frac{1-\zeta_{q_1^n}^a}{1-\zeta_{q_1^n}}\right)\\
		&\quad \quad -\sum_{\substack{1<t<q_2^n/2\\(t,q_2)=1}}\beta_t^{\prime} \log_p \left(\frac{1-\zeta_{q_2^n}^t}{1-\zeta_{q_2^n}}\right),
		\end{align*}
		\noindent
		where $\delta$, $\eta$, $\alpha_a^{\prime}$'s and $\beta_t^{\prime}$'s are algebraic numbers.\\
		\noindent
		By Lemma \ref{L1}, this is transcendental which is a contradiction.\\
		\noindent
		Now, again using Eq. (\ref{EQ7}) and Eq. (\ref{EQ6}) for $a \in S_1$ and $b \in S_2$, respectively, we have:
		\begin{align*}
		&\psi_p(r_1/p^n) +\gamma_p - \Big(\frac{p^{\mu}}{p^{\mu}-1}H^{\prime}_{\mu}(r_2/q^n) +\gamma_p\Big)\\
		& \quad = - \log_p p^n + \displaystyle\sum_{a=1}^{p^n - 1}{\zeta_{p^n}^{-ar_1}} \log_p (1- \zeta_{p^n}^{a}) + \log_p q^n - \sum_{t=1}^{q^n -1}{\zeta_{q^n}^{-tr_2}} \log_p (1- \zeta_{q^n}^t).
		\end{align*}
		\noindent
		Working on the similar lines and using $\log_p(p)=0$, we get:
		\begin{align*}
		&\sum_{\substack{1<a<p^n/2\\(a,p)=1}}{\alpha_a \log_p\Big(\frac{1- \zeta_{p^n}^{a}}{1-\zeta_{p^n}}\Big)} - \log_p(1-\zeta_{p^n}) + n \log_p \Big(\displaystyle\prod_{\substack{b=1\\(b,q)=1}}^{q^n - 1}{(1- \zeta_{q^n}^b)\Big)}\\
		&\quad \quad - \sum_{1<t<q^n/2}{(\zeta_{q^n}^{-tr_2} + \zeta_{q^n}^{tr_2}) \log_p\Big(\frac{1-\zeta^t_{q^n}}{1-\zeta_{q^n}}\Big)} -\log_p(1-\zeta_{q^n})\\
		&\quad = -\log_p(1- \zeta_{p^n})+ \delta \log_p(1- \zeta_{q^n}) + \displaystyle\sum_{\substack{1<a<p^n/2\\(a,p)=1}} \alpha_a^{\prime} \log_p \Big(\frac{1-\zeta_{p^n}^a}{1-\zeta_{p^n}}\Big) \\
		&\quad \quad \quad - \displaystyle\sum_{\substack{1<t<q^n/2\\(t,q)=1}}\beta_t^{\prime} \log_p \Big(\frac{1-\zeta_{q^n}^t}{1-\zeta_{q^n}}\Big),
		\end{align*}
		\noindent
		where $\delta$, $\alpha_a^{\prime}$'s, and $\beta_t^{\prime}$'s are algebraic numbers. Finally, using Lemma \ref{L1}, we conclude that it is transcendental, hence gives us a contradiction.\\
		For the second part of the proof, we take into account two scenarios - one in which $q$ is fixed and the other in which $q$ varies. 
		\begin{align*}
		\text{\textbf{Case 1:}}~~&\frac{p^{\mu}}{p^{\mu}-1}H^{\prime}_{\mu}(r_1/q) - \frac{p^{\mu}}{p^{\mu}-1}H^{\prime}_{\mu}(r_2/q) =  \displaystyle\sum_{a=1}^{q - 1}{\zeta_{q}^{-ar_1}} \log_p (1- \zeta_{q}^{a})\\
		& \quad - \sum_{t=1}^{q -1}{\zeta_{q}^{-tr_2}} \log_p (1- \zeta_{q}^t) = \sum_{1<a<q/2}(\zeta_q^{-ar_1} + \zeta_q^{ar_1} - \zeta_q^{-ar_2} - \zeta_q^{ar_2}) \log_p\Bigg(\frac{1-\zeta_q^{a}}{1-\zeta_q}\Bigg).
		\end{align*}
		Since $1 \leq r_1, r_2 < q/2$, the above linear form in logarithms is transcendental by Lemma~\ref{L1}. 
		\begin{align*}
		\text{\textbf{Case 2:}}~~ &\frac{p^{\mu_1}}{p^{\mu_1}-1}H^{\prime}_{\mu_1}(r_1/q_1) - \frac{p^{\mu_2}}{p^{\mu_2}-1}H^{\prime}_{\mu_2}(r_2/q_2) \\
		& \quad = - \log_p q_1 ~+ ~ \displaystyle\sum_{a=1}^{q_1 - 1}{\zeta_{q_1}^{-ar_1}} \log_p (1- \zeta_{q_1}^{a}) ~+~ \log_p q_2 ~-~ \sum_{t=1}^{q_2 -1}{\zeta_{q_2}^{-tr_2}} ~\log_p (1- \zeta_{q_2}^t)\\
		& \quad = -\displaystyle\sum_{b=1}^{q_1 - 1} \log_p (1- \zeta_{q_1}^{b})  + \log_p(1-\zeta_{q_1}) +  \sum_{1<a<q_1/2}(\zeta_{q_1}^{-ar_1} + \zeta_{q_1}^{ar_1}) \log_p\Bigg(\frac{1-\zeta_{q_1}^{a}}{1-\zeta_{q_1}}\Bigg) \\
		& \quad \quad + \displaystyle\sum_{s=1}^{q_2 - 1} \log_p (1- \zeta_{q_2}^{s}) - \log_p(1-\zeta_{q_2}) -  \sum_{1<t<q_2/2}(\zeta_{q_2}^{-tr_2} + \zeta_{q_2}^{tr_2}) \log_p\Bigg(\frac{1-\zeta_{q_2}^{t}}{1-\zeta_{q_2}}\Bigg)\\
		& \quad = \delta \log_p(1- \zeta_{q_1})+ \eta \log_p(1- \zeta_{q_2}) + \displaystyle\sum_{1<a<q_1/2} \alpha_a^{\prime} \log_p \left(\frac{1-\zeta_{q_1}^a}{1-\zeta_{q_1}}\right)\\
		&\quad \quad - \displaystyle\sum_{1<t<q_2/2}\beta_t^{\prime} \log_p \left(\frac{1-\zeta_{q_2}^t}{1-\zeta_{q_2}}\right),
		\end{align*}
		where $\delta$, $\eta$, $\alpha_a^{\prime}$'s, and $\beta_t^{\prime}$'s are algebraic numbers and it is transcendental by Lemma \ref{L1}. This completes the proof.
	\end{proof}
	
	\begin{proof}[Proof of Theorem \ref{TH6}]
		Consider the set
		$$S_3 = \{\psi_p(r/pq) + \gamma_p : 1 \leq r < pq,  (r,pq) = 1\}.$$
		Let there be two distinct algebraic elements in the set. Then by using Eq. (\ref{EQ7}), we have:
		\begin{align*}
		&\psi_p(r_1/pq) +\gamma_p - (\psi_p(r_2/pq) +\gamma_p)\\
		&\quad = - \log_p pq + \displaystyle\sum_{a=1}^{pq - 1}{\zeta_{pq}^{-ar_1}} \log_p (1- \zeta_{pq}^{a}) + \log_p pq - \sum_{t=1}^{pq -1}{\zeta_{pq}^{-tr_2}} \log_p (1- \zeta_{pq}^t)\\
		& \quad = \sum_{\substack{a=1\\ (a,pq)=1}}^{pq-1}\zeta_{pq}^{-ar_1}\log_p (1- \zeta_{pq}^{a}) + \sum_{\substack{a=1\\ (a,pq)=p}}^{pq-1}\zeta_{pq}^{-ar_1}\log_p (1- \zeta_{pq}^{a})\\
		&\quad \qquad + \sum_{\substack{a=1\\ (a,pq)=q}}^{pq-1}\zeta_{pq}^{-ar_1}\log_p (1- \zeta_{pq}^{a}) - \sum_{\substack{t=1\\ (t,pq)=1}}^{pq-1}\zeta_{pq}^{-tr_2}\log_p (1- \zeta_{pq}^{t})\\
		&\quad \qquad \qquad - \sum_{\substack{t=1\\ (t,pq)=p}}^{pq-1}\zeta_{pq}^{-tr_2}\log_p (1- \zeta_{pq}^{t}) - \sum_{\substack{t=1\\ (t,pq)=q}}^{pq-1}\zeta_{pq}^{-tr_2}\log_p (1- \zeta_{pq}^{t})\\
		&\quad = \sum_{\substack{1<a<pq/2\\ (a,pq)=1}}\alpha_a\log_p \Bigg(\frac{1- \zeta_{pq}^{a}}{1- \zeta_{pq}}\Bigg) + \sum_{1<b<q/2}\beta_b\log_p \Bigg(\frac{1- \zeta_{q}^{b}}{1- \zeta_{q}}\Bigg) + \sum_{1<c<p/2}\delta_c\log_p \Bigg(\frac{1- \zeta_{p}^{c}}{1- \zeta_{p}}\Bigg),
		\end{align*}
		where $\alpha_a$'s, $\beta_b$'s, and $\delta_c$'s are algebraic numbers and it is transcendental by Lemma \ref{L3}. This is a contradiction to our assumption that both are algebraic. Note that the proof also establishes the distinctness of the numbers $\psi_p(r/pq) + \gamma_p$, where $1 \leq r < pq/2$ and $(r,pq) = 1$. This completes the proof.
	\end{proof}
	\begin{proof}[Proof of Theorem \ref{TH3}]
		Consider the set
		\begin{align*}
		S_4 = \Bigg\{\frac{p^{\mu}}{p^{\mu}-1}H^{\prime}_{\mu}(r/m_i) +\gamma_p : 1 \leq r < m_i, ~1 \leq i \leq n,~ (r,m_i) = 1, p \nmid m_i, ~m_i \in \mathcal{M}\Bigg\},
		\end{align*}
		where $\mu$ as in Eq. (\ref{EQ6}). Let $m_1 = q_1^{a_1} q_2^{a_2}$ and $m_2 = q_3^{a_3} q_4^{a_4}$ be two distinct algebraic elements of $S_4$. Then, by using Eq. (\ref{EQ6}), we have:
		\begin{align}
		&\frac{p^{\mu_1}}{p^{\mu_1}-1}H^{\prime}_{\mu_1}(r_1/q_1^{a_1} q_2^{a_2}) +\gamma_p - \frac{p^{\mu_2}}{p^{\mu_2}-1}H^{\prime}_{\mu_2}(r_2/q_3^{a_3} q_4^{a_4}) -\gamma_p \nonumber\\
		&\quad = - \log_p (q_1^{a_1} q_2^{a_2}) + \displaystyle\sum_{a=1}^{q_1^{a_1} q_2^{a_2} - 1}{\zeta_{q_1^{a_1} q_2^{a_2}}^{-ar_1}} \log_p (1- \zeta_{q_1^{a_1} q_2^{a_2}}^{a}) + \log_p (q_3^{a_3} q_4^{a_4}) \nonumber\\
		& \quad \quad - \sum_{t=1}^{q_3^{a_3} q_4^{a_4} -1}{\zeta_{q_3^{a_3} q_4^{a_4}}^{-tr_2}} \log_p (1- \zeta_{q_3^{a_3} q_4^{a_4}}^t) \nonumber \\
		&\quad  = -a_1\log_p q_1 -a_2\log_p q_2+ \sum_{\substack{a=1\\ (a,q_1 q_2)=1}}^{q_1^{a_1} q_2^{a_2}-1}\zeta_{q_1^{a_1} q_2^{a_2}}^{-ar_1}\log_p (1- \zeta_{q_1^{a_1} q_2^{a_2}}^{a}) \nonumber \\
		& \quad \quad + \sum_{\substack{a=1\\ (a,q_1 q_2) \neq 1}}^{q_1^{a_1} q_2^{a_2}-1}\zeta_{q_1^{a_1} q_2^{a_2}}^{-ar_1}\log_p (1- \zeta_{q_1^{a_1} q_2^{a_2}}^{a}) + a_3\log_p q_3 + a_4\log_p q_4 \nonumber \\
		&\quad \quad \quad - \sum_{\substack{t=1\\ (t,q_3 q_4)=1}}^{q_3^{a_3} q_4^{a_4}-1}\zeta_{q_3^{a_3} q_4^{a_4}}^{-tr_2}\log_p (1- \zeta_{q_3^{a_3} q_4^{a_4}}^{t}) - \sum_{\substack{t=1\\ (t,q_3 q_4) \neq 1}}^{q_3^{a_3} q_4^{a_4}-1}\zeta_{q_3^{a_3} q_4^{a_4}}^{-tr_2}\log_p (1- \zeta_{q_3^{a_3} q_4^{a_4}}^{t})\nonumber\\
		& \quad =-\displaystyle a_1\log_p \Bigg( \prod_{b=1}^{q_1 - 1}(1 - \zeta_{q_1}^b)\Bigg)-a_2\log_p \Bigg( \prod_{d=1}^{q_2 - 1}(1 - \zeta_{q_2}^d)\Bigg)-\log_p(1- \zeta_{q_1^{a_1} q_2^{a_2}}) \nonumber\\
		& \quad \quad + \sum_{\substack{1< a<q_1^{a_1} q_2^{a_2}/2\\ (a,q_1 q_2)=1}}\alpha_a \log_p \Bigg(\frac{1- \zeta_{q_1^{a_1} q_2^{a_2}}^{a}}{1- \zeta_{q_1^{a_1} q_2^{a_2}}}\Bigg) + \sum_{\substack{1<a<q_1^{a_1} q_2^{a_2}/2\\ (a,q_1 q_2) \neq 1}}\delta_a\log_p \Bigg(\frac{1- \zeta_{q_1^{a_1} q_2^{a_2}}^{a}}{1- \zeta_{q_1^{a_1} q_2^{a_2}}}\Bigg) \nonumber\\
		& \quad \quad \quad + \displaystyle a_3\log_p \Bigg( \prod_{r=1}^{q_3 - 1}(1 - \zeta_{q_3}^r)\Bigg) + a_4\log_p \Bigg( \prod_{s=1}^{q_4 - 1}(1 - \zeta_{q_4}^s)\Bigg) + \log_p(1- \zeta_{q_3^{a_3} q_4^{a_4}}) \nonumber\\
		&\quad \quad \quad \quad - \sum_{\substack{1<t<q_3^{a_3} q_4^{a_4}/2\\ (t,q_3 q_4)=1}}\beta_a \log_p \Bigg(\frac{1- \zeta_{q_3^{a_3} q_4^{a_4}}^{t}}{1- \zeta_{q_3^{a_3} q_4^{a_4}}}\Bigg) - \sum_{\substack{1<t<q_3^{a_3} q_4^{a_4}/2\\ (t,q_3 q_4) \neq 1}}\gamma_a\log_p \Bigg(\frac{1- \zeta_{q_3^{a_3} q_4^{a_4}}^{t}}{1- \zeta_{q_3^{a_3} q_4^{a_4}}}\Bigg) \nonumber\\
		&\quad = a_1\log_p(1-\zeta_{q_1}) + a_2\log_p(1-\zeta_{q_2}) - a_3\log_p(1-\zeta_{q_3}) -a_4\log_p(1-\zeta_{q_4}) \nonumber\\
		& \quad \quad -\log_p(1- \zeta_{q_1^{a_1} q_2^{a_2}}) + \log_p(1- \zeta_{q_3^{a_3} q_4^{a_4}})- \sum_{1<b<q_1/2}\alpha_b \log_p \Bigg(\frac{1- \zeta_{q_1 }^{b}}{1- \zeta_{q_1}}\Bigg) \nonumber \\
		& \quad \quad \quad - \sum_{1<d<q_2/2}\delta_d \log_p \Bigg(\frac{1- \zeta_{q_2 }^{d}}{1- \zeta_{q_2}}\Bigg) + \sum_{1<r<q_3/2}\beta_r \log_p \Bigg(\frac{1- \zeta_{q_3 }^{r}}{1- \zeta_{q_3}}\Bigg) \nonumber \\
		& \quad \quad \quad \quad + \sum_{1<s<q_4/2}\gamma_s \log_p \Bigg(\frac{1- \zeta_{q_4}^{s}}{1- \zeta_{q_4}}\Bigg) + \sum_{\substack{1< a<q_1^{a_1} q_2^{a_2}/2\\ (a,q_1 q_2)=1}}\alpha_a \log_p \Bigg(\frac{1- \zeta_{q_1^{a_1} q_2^{a_2}}^{a}}{1- \zeta_{q_1^{a_1} q_2^{a_2}}}\Bigg) \nonumber\\
		& \quad \quad \quad \quad \quad + \sum_{\substack{1< a<q_1^{a_1} q_2^{a_2}/2\\ (a,q_1 q_2) \neq 1}}\delta_a\log_p \Bigg(\frac{1- \zeta_{q_1^{a_1} q_2^{a_2}}^{a}}{1- \zeta_{q_1^{a_1} q_2^{a_2}}}\Bigg)  - \sum_{\substack{1<t<q_3^{a_3} q_4^{a_4}/2\\ (t,q_3 q_4)=1}}\beta_a \log_p \Bigg(\frac{1- \zeta_{q_3^{a_3} q_4^{a_4}}^{t}}{1- \zeta_{q_3^{a_3} q_4^{a_4}}}\Bigg) \nonumber \\
		& \quad \quad \quad \quad \quad \quad - \sum_{\substack{1<t<q_3^{a_3} q_4^{a_4}/2\\ (t,q_3 q_4) \neq 1}}\gamma_a\log_p \Bigg(\frac{1- \zeta_{q_3^{a_3} q_4^{a_4}}^{t}}{1- \zeta_{q_3^{a_3} q_4^{a_4}}}\Bigg), \label{EQ5}
		\end{align}
		where $a_1$, $a_2$, $a_3$, $a_4$, $\alpha_b$'s, $\delta_d$'s, $\beta_r$'s, $\gamma_s$'s, $\alpha_a$'s, $\delta_a$'s, $\beta_a$'s, and $\gamma_a$'s are algebraic numbers. By applying Theorem \ref{TH5} to the maximal linearly independent set of logarithm of algebraic terms in Eq. (\ref{EQ5}), we find that it is transcendental, which leads to a contradiction.
	\end{proof}
	
	\begin{proof}[Proof of Corollary \ref{C1}]
		Let $S^{\prime} = S_3 \cup S_4$. Suppose that $a, b \in S^{\prime}$ are distinct and algebraic. Then, the possibilities of $(a,b)$ are:
		\begin{enumerate}
			\item $a \in S_3$ and $b \in S_3$.
			\item $a \in S_4$ and $b \in S_4$.
			\item $a \in S_3$ and $b \in S_4$.
		\end{enumerate}
		Theorem \ref{TH6} and Theorem \ref{TH3} already address the $1st$ and $2nd$ case, respectively. Therefore, we consider the scenario when $a \in S_3$ and $b \in S_4$. Using Eq. (\ref{EQ7}) and Eq. (\ref{EQ6}), we have:
		\begin{align}
		&\psi_p(r_1/pq) +\gamma_p - \Big(\frac{p^{\mu}}{p^{\mu}-1}H^{\prime}_{\mu}(r_2/q_1^mq_2^n) +\gamma_p\Big) \nonumber\\
		&\quad  = - \log_p pq + \displaystyle\sum_{a=1}^{pq - 1}{\zeta_{pq}^{-ar_1}} \log_p (1- \zeta_{pq}^{a}) + \log_p q_1^mq_2^n - \sum_{t=1}^{q_1^mq_2^n -1}{\zeta_{q_1^mq_2^n}^{-tr_2}} \log_p (1- \zeta_{q_1^mq_2^n}^t) \nonumber\\
		& \quad = -\log_pq + \sum_{\substack{a=1\\ (a,pq)=1}}^{pq-1}\zeta_{pq}^{-ar_1}\log_p (1- \zeta_{pq}^{a}) + \sum_{\substack{a=1\\ (a,pq)=p}}^{pq-1}\zeta_{pq}^{-ar_1}\log_p (1- \zeta_{pq}^{a}) \nonumber \\
		&\quad \quad + \sum_{\substack{a=1\\ (a,pq)=q}}^{pq-1}\zeta_{pq}^{-ar_1}\log_p (1- \zeta_{pq}^{a}) +m\log_pq_1 + n\log_p q_2 -\sum_{\substack{t=1\\ (t,q_1 q_2)=1}}^{q_1^{m} q_2^{n}-1}\zeta_{q_1^{m} q_2^{n}}^{-tr_2}\log_p (1- \zeta_{q_1^{m} q_2^{n}}^{t}) \nonumber\\
		&\quad \quad \quad - \sum_{\substack{t=1\\ (t,q_1 q_2) \neq 1}}^{q_1^{m} q_2^{n}-1}\zeta_{q_1^{m} q_2^{n}}^{-tr_2}\log_p (1- \zeta_{q_1^{m} q_2^{n}}^{t}). \nonumber
		\end{align}
		After simplification of the terms, we get:
		\begin{align}
		& -\alpha_1 \log_p(1-\zeta_q)  + \log_p(1-\zeta_{pq})+ \sum_{\substack{1<a<pq/2\\ (a,pq)=1}}\alpha_a\log_p \Bigg(\frac{1- \zeta_{pq}^{a}}{1- \zeta_{pq}}\Bigg) \nonumber\\
		& \quad + \sum_{1<b<q/2}\beta_b\log_p \Bigg(\frac{1- \zeta_{q}^{b}}{1- \zeta_{q}}\Bigg) + \sum_{1<c<p/2}\delta_c\log_p \Bigg(\frac{1- \zeta_{p}^{c}}{1- \zeta_{p}}\Bigg)  +\alpha_2\log_p(1-\zeta_{q_1}) \nonumber\\
		& \quad \quad + \alpha_3\log_p(1-\zeta_{q_2}) - \log_p(1-\zeta_{q_1^{m}q_2^{n}}) + \sum_{1<b<q_1/2}\alpha_b \log_p \Bigg(\frac{1- \zeta_{q_1 }^{b}}{1- \zeta_{q_1}}\Bigg) \nonumber \\
		&\quad \quad \quad + \sum_{1<d<q_2/2}\delta_d \log_p \Bigg(\frac{1- \zeta_{q_2 }^{d}}{1- \zeta_{q_2}}\Bigg)- \sum_{\substack{1<t< q_1^{m} q_2^{n}/2\\ (t,q_1 q_2)=1}}\alpha_t \log_p \Bigg(\frac{1- \zeta_{q_1^{m} q_2^{n}}^{a}}{1- \zeta_{q_1^{m} q_2^{n}}}\Bigg) \nonumber \\
		& \quad \quad \quad \quad + \sum_{\substack{1< t< q_1^{m} q_2^{n}/2 \\ (t,q_1 q_2) \neq 1}}\delta_t\log_p \Bigg(\frac{1- \zeta_{q_1^{m} q_2^{n}}^{t}}{1- \zeta_{q_1^{m} q_2^{n}}}\Bigg), \label{EQ8}
		\end{align}
		where $\alpha_1$, $\alpha_2$, $\alpha_3$, $\alpha_a$'s,$\alpha_b$'s, $\alpha_t$'s, $\beta_b$'s, $\delta_c$'s, $\delta_d$'s, and $\delta_t$'s are algebraic numbers. By applying Theorem \ref{TH5} to the maximal linearly independent set of the logarithm of algebraic numbers in Eq. (\ref{EQ8}), we find that it is transcendental, which leads to a contradiction.
	\end{proof}
	
	\begin{proof}[Proof of Theorem \ref{TH7}]
		Consider the sets
		$$S_5=\{\psi_p(r/q) + \gamma_p : 1 \leq r < q,  (r,q) = 1 \}$$ 
		and
		$$S_6=\Bigg\{\frac{p^{\mu}}{p^{\mu}-1}H^{\prime}_{\mu}(r/q) +\gamma_p : 1 \leq r < q, (r,q) = 1 \Bigg\}.$$
		Firstly, we discuss the case when $p \mid q$. Let there be two algebraic and distinct elements of the set $S_5$. Then by Eq. (\ref{EQ7}), we have:
		\begin{align*}
		&\psi_p(r_1/q) +\gamma_p - (\psi_p(r_2/q) +\gamma_p)\\
		&\quad = - \log_p q + \displaystyle\sum_{a=1}^{q - 1}{\zeta_{q}^{-ar_1}} \log_p (1- \zeta_{q}^{a}) + \log_p q - \sum_{t=1}^{q -1}{\zeta_{q}^{-tr_2}} \log_p (1- \zeta_{q}^t)\\
		&\quad = \sum_{\substack{1<a<q/2\\ (a,q)=1}}\alpha_a\log_p \Bigg(\frac{1- \zeta_{q}^{a}}{1- \zeta_{q}}\Bigg),
		\end{align*}
		where $\alpha_a$'s are algebraic numbers. Then, using Proposition \ref{P5} along with Theorem \ref{TH5}, we get that this is a transcendental number and hence it gives us a contradiction.\\
		\noindent
		For the second case when $p \nmid q$, using Eq. (\ref{EQ6}) and proceeding along the similar lines as in the above case, we get the desired outcomes. 
	\end{proof}
	
	\section{Concluding Remarks}
	Theorem \ref{TH6} can be extended for the set $$S_7 = \{\psi_p(r/p^aq^b) + \gamma_p : 1 \leq r < p^aq^b,  (r,pq) = 1\} $$ where $p, q \in \mathcal{J}$ and $m = p^a q^b \in \mathcal{M}$. The proof of this follows a similar approach as of Theorem \ref{TH6} and eventually the result follows by taking the maximal linearly independent set of logarithms of algebraic terms. 
	
	\section{Acknowledgements}
	The authors thank Michel Waldschmidt for some valuable suggestions on an earlier version of the paper. Also, we would like to thank SERB and UGC for providing the partial support.

\end{document}